\documentclass[11pt]{amsart}
\usepackage[active]{srcltx}

\usepackage{amsmath, amsfonts,amsthm,times,graphics,color}
\newcommand{\vertiii}[1]{{\left\vert\kern-0.25ex\left\vert\kern-0.25ex\left\vert #1
    \right\vert\kern-0.25ex\right\vert\kern-0.25ex\right\vert}}
 \makeatletter
\renewcommand*\subjclass[2][2000]{%
  \def\@subjclass{#2}%
  \@ifundefined{subjclassname@#1}{%
    \ClassWarning{\@classname}{Unknown edition (#1) of Mathematics
      Subject Classification; using '1991'.}%
  }{%
    \@xp\let\@xp\subjclassname\csname subjclassname@#1\endcsname
  }%
}
 \makeatother
\usepackage{enumerate,amssymb,  mathrsfs,yhmath}

\newtheorem{theorem}{Theorem}[section]
\newtheorem{lemma}[theorem]{Lemma}
\newtheorem*{lemma*}{Lemma}
\newtheorem{proposition}[theorem]{Proposition}

\newtheorem{definition}[theorem]{Definition}

\def\1ton{1,2,\ldots,n}

\usepackage{amssymb}
\usepackage{amsthm}
\usepackage{mathrsfs, amsfonts, amsmath}
\usepackage{graphicx}

\newcommand{\R}{\mathbb{R}}

\newtheorem{remark}[theorem]{Remark}

\numberwithin{equation}{section}

\renewcommand{\imath}{i} 

\def\XXint#1#2#3{{\setbox0=\hbox{$#1{#2#3}{\int}$}
\vcenter{\hbox{$#2#3$}}\kern-.5\wd0}}

\def\ge{\geqslant}
\begin{document}

\title{Contraction property of  Fock type space of log-subharmonic functions in $\mathbb{R}^m$}

\author{David Kalaj}
\address{University of Montenegro, Faculty of Natural Sciences and
Mathematics, Cetinjski put b.b. 81000 Podgorica, Montenegro}
\email{davidk@ucg.ac.me}

\begin{abstract}

We prove a contraction property of Fock type spaces $\mathcal{L}_{\alpha}^p$ of log-subharmonic functions in $\mathbb{R}^n$.  To prove the result, we demonstrate a certain monotonic property of measures of the superlevel set of the function $u(x) = |f(x)|^p e^{-\frac{\alpha}{2} p |x|^2}$, provided that $f$ is a certain log-subharmonic function in $\mathbb{R}^m$.  The result recover a contraction property of holomorphic functions in the Fock space $\mathcal{F}_\alpha^p$ proved by Carlen in \cite{carlen}.

\end{abstract}
\maketitle
\tableofcontents
\sloppy

\maketitle
\section{Introduction}

Let $m\ge 1$ and let $\mathbb{R}^m$ be the Euclidean space endowed with the Euclidean norm: $|x|=\sqrt{\left<x,x\right>}$, where $\left<x,y\right>=\sum_{i=1}^m x_i y_i$, and $x=(x_1,\dots, x_m), y=(y_1,\dots,y_n)\in \mathbb{R}^m$. If $\alpha>0$ and $p>0$ and $m=2n$ is an even integer, we define the Fock space or Segal-Bargmann space $\mathcal{F}_\alpha^p$ (cf. \cite{bargman,bargman1, kezu}) of  entire holomorphic functions $f$ in $\mathbb{C}^n=\mathbb{R}^{2n}$ so that:
\begin{equation}\label{fockno}\|f\|^p_{p, \alpha}: =c_{p, \alpha}\int_{\mathbb{R}^m} |f(x)|^p e^{-\frac{\alpha}{2} p |x|^2} dA(x)<\infty,$$ where $$c_{p, \alpha}=\left(\frac{\alpha p}{2\pi}\right)^{\frac{m}{2}},\end{equation} and $dA(x)$ is Lebesgue measure on $\mathbb{R}^m$. Note that $c_{p,\alpha}e^{-\frac{\alpha}{2} p |x|^2} dA(x)$ is the Gaussian probabily measure in $\mathbb{R}^m$.

Assume now that $m\in\mathbb{N}$ is an arbitrary integer.
 We say that a real twice differentiable function $f$ defined in a domain $\Omega\subset \mathbb{R}^m$ is subharmonic if $\Delta  f(x)\ge 0$ for $x\in \Omega$. Here $\Delta$ is the Laplacian. This definition can also be extended to not necessary double differentiable functions, by using the sub-mean value property (\cite{hake}). We say that a mapping $f$ is log-subharmonic, if $\log |f(x)|$ is subharmonic in $\Omega\setminus f^{-1}(0)$. We denote by $\mathcal{L}^p_\alpha$ the space of complex-valued, real-analytic functions whose absolute value is a log-subharmonic function, defined in $\mathbb{R}^m$, with a finite $\|f\|_{p,\alpha}$ norm as defined in \eqref{fockno}. Here $m$ is an arbitrary positive integer. Observe that for $m=2n$ we have $\mathcal{F}_\alpha^p\subset \mathcal{L}_\alpha^p$:  If $f$ is holomorphic in $\Omega$, then $ |f(z)|$ is log-subharmonic. Indeed  $$\Delta \log |f(z)|=\sum_{k=1}^n \Delta_{z_k} \log |f(z)|=0,$$ where $z=(z_1,\dots, z_n)$, and $$\Delta_{z_k}=\frac{\partial^2}{(\partial_{x_k})^2}+ \frac{\partial^2}{(\partial_{y_k})^2},$$ $z_k=x_k+\imath y_k$ for $k=1,\dots,n$ and $z\in \Omega\setminus f^{-1}(0).$

\section{Motivation and the Main Results}
Carlen, in his paper \cite{carlen} proved the following result:
\begin{theorem}\label{kehezu} If $0<p<q<\infty$, then $\mathcal{F}^p_\alpha(\mathbb{C}^n)\subset \mathcal{F}^q_\alpha(\mathbb{C}^n)$
 and the inclusion is proper and
continuous. Moreover $$\|f\|_{q, \alpha}\le \|f\|_{p, \alpha}.$$
\end{theorem}
Theorem~\ref{kehezu} is applied in \cite{carlen} to the coherent state transform in a new proof of Wehrl's entropy conjecture \cite{state}. In this paper, among other results, we recover Theorem~\ref{kehezu} and provide a proof that works for a more general class of mappings, namely real analytic complex mappings whose absolute value is a log-subharmonic function in $\mathbb{R}^m$ and belongs to the Fock-type space $\mathcal{L}^p_\alpha$.

Let $f$ be a real analytic complex-valued function defined in the Euclidean space $\mathbb{R}^m$, such that  $v=|f|$ is a log-subharmonic function in $\mathbb{R}^m$  and such that $u(x) = v(x)^p e^{-\alpha p/2|x|^2}$ is bounded and goes to $0$
uniformly as $|x|\to \infty$. Then the superlevel sets $A_t = \{x: u(x) > t\}$ for $t > 0$ are
compactly embedded in $\mathbb{R}^m$ and thus have finite Lebesgue  measure $\mu(t) = |A_t|$.

Those are the main results
\begin{theorem}\label{monotone}
Let $\alpha>0$ and $p> 0$ and assume that $f$ is a real analytic complex valued function such that $v=|f|:\mathbb{R}^m \to [0,+\infty)$ is a log-subharmonic function. Assume further that the function $u(x) = |f(x)|^p e^{-\frac{\alpha p}{2} |x|^2}$ is bounded and $u(x)$ tends to $0$ uniformly as $|x|\to \infty$. Then the function $$g(t) =t\exp\left[\frac{\alpha p (\Gamma(m/2))^{2/m}}{2\pi}\mu^{2/m}(t) \right],$$ is decreasing on the interval $(0, t_\circ)$, where  $t_\circ = \max_{x\in\mathbb{R}^m} u(x)$.
\end{theorem}

If  $f(x) \equiv 1$, the function $g$ turns out to be constant and this is an important property of $g$.

The proof of this theorem is mostly based on the methods developed by Nicola and Tilli in \cite{3} (see also the subsequent papers where similar methods are used: \cite{kulik}, \cite{kalaj2024}, \cite{blms}, \cite{KNOT}, and \cite{frank}).

By using Theorem~\ref{monotone} we will prove the following theorem:

\begin{theorem}\label{Fock}
Let $p>0$ and $\alpha>0$. Let $G:[0,\infty)\to \R$ be a convex function. Then the maximum value of
\begin{equation}\label{bergmanineq}
\int_{\mathbb{R}^m} G(|f(x)|^pe^{-\frac{\alpha}{2} p|x|^2})dA(x)
\end{equation}
is attained for $$f_a(x) =e^{\alpha \left<a,x\right>-\frac{\alpha}{2} |a|^2},$$ where $a\in\mathbb{C}^n$ is arbitrary, subject to the condition that $f\in \mathcal{L}^p_\alpha$ and $ \|f \|_{p, \alpha}=1$.
\end{theorem}

Applying Theorem~\ref{Fock} to the convex and increasing function
$G(t)=t^{q/p}$, we get the extension of theorem \cite[Theorem~2]{carlen} by proving:
\begin{theorem}\label{shtune}
For all $0 < p < q < \infty$ and $0 < \alpha$ and for $f\in \mathcal{L}^p_\alpha(\mathbb{R}^m)$,  we have $f\in \mathcal{L}^q_\alpha(\mathbb{R}^m)$ and
$$\|f\|_{q,\alpha} \le \|f\|_{p, \alpha}$$
with equality for $f_a(x) =e^{\alpha \left<a,x\right>-\frac{\alpha}{2} |a|^2}$, where  $a\in \mathbb{R}^m$ is arbitrary.

\end{theorem}
\begin{proof}[Proof of Theorem~\ref{shtune}] For $ \|f \|_{p, \alpha}=N$, $\|f/N\|_{p,\alpha}=1$ and from Theorem~\ref{Fock} we have $$\int_{\mathbb{R}^m} |f(x)/N|^qe^{-\frac{\alpha}{2} q|x|^2}dA(x)\le \int_{\mathbb{R}^m} e^{-\frac{\alpha}{2} q|x|^2}dA(x) =1/c_{q,\alpha}.$$ Thus $$c_{q,\alpha}\int_{\mathbb{R}^m} |f(x)|^qe^{-\frac{\alpha}{2} q|x|^2}dA(x)\le N^q,$$ or what is the same $$\|f\|_{q,\alpha} \le \|f\|_{p, \alpha}.$$ The equality statement follows from the equality statement of Theorem~\ref{shtune}, but can be proved by using the same approach as in monograph of Zhu \cite[Lemma~2.33]{kezu}.
\end{proof}
\begin{remark}
The last theorem is an extension of Theorem~\ref{kehezu}. Moreover, its proof is different from the proof in \cite{carlen} and seems to be simpler. We refer to the paper \cite{jfa} for some related inequalities for log-subharmonic functions in $\mathbb{R}^n$.
\end{remark}

Theorem~\ref{shtune} is a counterpart of a similar contraction property of Bergman spaces $\mathbf{B}^p_\alpha$ (\cite[p.~2]{hakan}), proved by Kulikov in \cite{kulik} for holomorphic functions in the unit disk and for $\mathcal{M}-$log-subharmonic functions in the unit ball in $\mathbb{R}^n$ by the author in \cite{kalaj2024}. It is known that $$\mathbf{B}^p_\alpha \subset \mathbf{B}^q_\beta,\quad \frac{p}{\alpha} = \frac{q}{\beta} = r,\quad p < q.$$ For $n=2$, it was asked whether these embeddings are  contractions, that is whether the norm $  \|f \|_{\mathbf{B}^{r\alpha}_\alpha}$ is decreasing in $\alpha$. In the case of Bergman spaces, this question was asked by Lieb and Solovej \cite{9}. They proved that such contractivity implies their Wehrl-type entropy conjecture for the $SU(1,1)$ group. In the case of contractions from the Hardy spaces to the Bergman spaces, it was asked by Pavlovi\'c in \cite{13} and by Brevig, Ortega-Cerd\`a, Seip, and Zhao \cite{7} concerning the estimates for analytic functions. The mentioned contraction property proved by Kulikov confirmes these conjectures. An interesting application of Kulikov result has been given by Melentijevi\'c in \cite{petar}.

We end this paper with the construction of a new normed Fock type space:
\begin{definition}[Fock limit space]
Let $f$ be a holomorphic function in $\mathbb{C}^n$. Then for $\alpha>0$ we say $f\in \mathcal{F}_{\alpha}$ if $\displaystyle f\in \bigcap_{p>0} \mathcal{F}_\alpha^{ p}$. Then we define $$\|f\|_\alpha:=\inf_{p> 0} \|f\|_{p, \alpha}.$$
\end{definition}For $\alpha>0$ define as in \cite[eq.~2.2]{kezu} the following Banach norm $$\|f\|_{\infty,\alpha} :=\mathrm{esssup}\{|f(z)|e^{-\frac{\alpha}{2} |z|^2}, z\in\mathbb{C}^n\}.$$
Then we prove
\begin{theorem}\label{fockBanah}
For every $\alpha>0$ we have $$\|f\|_\alpha=\|f\|_{\infty, \alpha}.$$ In particular $(\mathcal{F}_\alpha, \|\cdot \|_\alpha)$ is a normed subspace of Banach space $\mathcal{F}_\alpha^\infty.$
\end{theorem}

\section{Proof of Theorem \ref{monotone}}
\begin{proof}[Proof of Theorem \ref{monotone}]
We start with the formula
$$\mu(t)=|A_t|=\int_{A_t}dx=\int_t^{\max u}\int_{|u(x)|=\kappa} d\mathcal{H}^{m-1}(x) d\kappa.$$ Then we get
\begin{equation}\label{derlevel}
-\mu'(t) = \int_{u = t}|\nabla u|^{-1}d\mathcal{H}^{m-1}(x)
\end{equation}
along with the claim that $\{x:u(x) = t\} = \partial A_t$ and that this set is a smooth hypersurface for almost all $t\in (0, t_\circ)$. Here $dS=d\mathcal{H}^{m-1}$ is $m-1$ dimensional Hausdorff measure.
These assertions follow the proof of \cite[Lemma~3.2]{3}. We point out that, since $u$ is real analytic, then it is a well-known fact from measure theory that the level set $\{x:u(x) = t\}$ has a zero measure (\cite{zero}), and this is equivalent to the fact that the $\mu$ is continuous.

Following the approach from \cite{3}, our next step is to apply the Cauchy--Schwarz inequality to the $m-1$ dimensional measure of $\partial A_t$:
\begin{equation}\label{CS}|\partial A_t|^2 = \left(\int_{\partial A_t} dS\right)^{2}\le \int_{\partial A_t} |\nabla u|^{-1} dS\int_{\partial A_t} |\nabla u| dS.
\end{equation}

Let $\nu=\nu(x)$ be the outward unit normal to $\partial A_t$ at a point $x$. Note that, $\nabla u$ is parallel to $\nu$, but directed in the opposite direction. Thus we have $|\nabla u| = -\left<\nabla u,\nu\right>$. Also, we note that since for $x\in \partial A_t$ we have $u(x) = t$, we obtain for $x\in \partial A_t$ that
$$\frac{|\nabla u(x)|}{t} = \frac{|\nabla u(x)|}{u} =  \left<\nabla  \log u(x), \nu\right>.$$

Now the second integral on the right-hand side of \eqref{CS} can be evaluated by Gauss's divergence theorem:
\[\begin{split}\int_{\partial A_t} {|\nabla u| |dS|}&=-t  \int_{A_t}\mathrm{div}\left({\nabla\log u(x)}\right) dA(x)
\\&=-t  \int_{A_t}\Delta {\log u(x)} dA(x).\end{split}\]
Now we plug
$u= |f(x)|^pe^{-\frac{\alpha}{2} p |x|^2} $, and calculate $$-t \Delta \log(|f(x)|^p e^{-\frac{\alpha}{2} p |x|^2})=-(p t\Delta \log v-t\frac{\alpha}{2} p \Delta |x|^2)\le 0+mt \alpha p. $$ By using \eqref{derlevel} and \eqref{CS} we obtain
\[\begin{split}|\partial A_t|^2 & \le  (-\mu'(t))\int_{\partial A_t} {|\nabla u|dS}.
\\&\le - m t \alpha p \mu'(t) \mu(t). \end{split}\]
Now we use the isoperimetric inequality for the space: $$ |\partial A_t|^2\ge \pi m^2 |A_t|^{\frac{2(m-1)}{m}}\left({\Gamma(m/2)}\right)^{-\frac{2}{m}}, $$
which implies that
\begin{equation}\label{ypsibaraz} mt \alpha p \mu'(t)\mu(t)+m^2 \pi \mu(t)^{\frac{2(m-1)}{m}}\left({\Gamma(m/2)}\right)^{-\frac{2}{m}}\le 0\end{equation}
with  equality in \eqref{ypsibaraz} if and only if $v(x)=e^{\alpha\left<x,a\right>-\frac{\alpha}{2}|a|^2}$ because in that case $A_t$ is a ball centered at $a$. So
\begin{equation}\label{ypsibaraza}  M(t):=\alpha p \mu'(t)\mu(t)^{\frac{2-m}{m}}+\frac{m\pi  \left({\Gamma(m/2)}\right)^{-\frac{2}{m}}}{t}\le 0.\end{equation}
Since   $\mu(t^\circ)=0$, we obtain that $$G(t)=\int_{t_\circ}^t M(t)dt= m \pi  (\Gamma(m/2))^{-2/m} \log \frac{t}{t_\circ}+\frac{m}{2} \alpha p  \mu^{\frac{2}{m}}(t)$$ is a non-increasing function  for $0\le t<t_\circ$.

 In the case $v(x)\equiv  e^{\alpha\left<a,x\right>-\frac{\alpha}{2}|a|^2}$,  $t_\circ=1$ and $\mu(t_\circ)=0$. Moreover $$g(t):=\exp(G(t)) =t\exp\left[\frac{\alpha p (\Gamma(m/2))^{2/m}}{2\pi}\mu^{2/m}(t) \right]$$ is non-increasing for $0\le t<t_\circ$.

\end{proof}
\begin{remark}
{Note that for the function $f(x) \equiv 1$ or $$f(x) = e^{-\frac{\alpha}{2}  |a|^2} e^{\alpha   \left<a, x\right>},$$ for a fixed $a$, everywhere in the proof above we have  equalities for all values of $p$ and $\alpha$.} Moreover in this case the maximum of $u(x)$ is equal to $1$ and achieved for $x=a$.
\end{remark}

\section{Proof of Theorem \ref{Fock}}\label{sec4}
We need the following lemma \begin{lemma}\label{lemm}\cite{kalaj2024}
Assume that $\Phi, \Psi$ are positive increasing functions and $g$ positive non-increasing such that
$$\int_0^{t_\circ} \Phi\left({g(t)}/{t}\right)dt = \int_0^{t_\circ} \Phi\left({1}/{t}\right)dt=c.$$

Then $$\int_0^{t_\circ} \Phi\left({g(t)}/{t}\right)\Psi(t)dt\le \int_0^{t_\circ} \Phi\left({1}/{t}\right)\Psi(t)dt.$$
\end{lemma}
As in \cite{kulik, kalaj2024}, where is treated Bergman version of this theorem, we restrict ourselves to the only nontrivial case $\lim_{t\to 0^{+}}G(t)  = 0$.
Let $\mu(t) = \mu(\{ x: u(x) > t\})$ be the Lebesgue measure in $\mathbb{R}^m$, where $u(x) = |f(x)|^pe^{-\frac{\alpha p}{2} |x|^2} $.
Applying Theorem \ref{monotone} to $f$, we get that the function
$$g(t) =t\exp\left[\frac{\alpha (\Gamma(m/2))^{2/m}}{2\pi}\mu^{2/m}(t)  \right],$$
is decreasing on $(0, t_\circ)$ with $t_\circ = \max_{x\in \mathbb{R}^m}u(x)$. Proposition~\ref{ulqin} below ensures  the existence of $t_\circ$.

For $f\equiv 1$, $g$ is a constant function equal to $1$.

Then $$\mu(t)=\left(\frac{2\pi}{\alpha (\Gamma(m/2))^{2/m}}\log \frac{g(t)}{t}\right)^{\frac{m}{2}}.$$
We assume that $ \|f \|_{p, \alpha} = 1$, that is
$$I_1=c_{p, \alpha}\int_0^{t_\circ} \mu(t)dt =c_{p, \alpha}\int_0^{t_\circ}\left(\frac{2\pi}{\alpha (\Gamma(m/2))^{2/m}}\log \frac{g(t)}{t}\right)^{m/2}dt= 1.$$
Now the integral in \eqref{bergmanineq} can be rewritten as
$$I_2=c_{p, \alpha}\int_0^{t_\circ} \left(\frac{2\pi}{\alpha (\Gamma(m/2))^{2/m}}\log \frac{g(t)}{t}\right)^{m/2}G'(t)dt.$$
Then by  Lemma~\ref{lemm}, by taking $\Phi(s) = c_{p, \alpha}\left(\frac{2\pi}{\alpha (\Gamma(m/2))^{2/m}}\log s\right)^{\frac{m}{2}}$ and $\Psi(t)=G'(t)$,  the maximum of $I_2$ under $I_1=1$ is attained for $g\equiv 1$.

\section{Some additional properties of Fock space and proof of Theorem~\ref{fockBanah}}
Now we prove the following proposition used in the proof of our main result.
\begin{proposition}
 \label{ulqin}
Assume that $f$ is a real-analytic log-subharmonic function in $\mathbb{R}^m$ belonging to the Fock type space. Then for every $x$,  \begin{equation}\label{sharine}|f(x)|^pe^{-\frac{\alpha p}{2} |x|^2}\le c_{p, \alpha} \int_{\mathbb{R}^m} |f(y)|^p e^{-\frac{\alpha p}{2} |y|^2} dA(y).\end{equation}
Moreover \begin{equation}\label{shark}\lim_{|x|\to \infty} |f(x)|e^{-\frac{\alpha}{2} |x|^2}=0.\end{equation}

\end{proposition}
Notice that \eqref{sharine} extends \cite[Theorem~2.7]{kezu} and the relation \eqref{shark} extends corresponding relation in \cite[p.~38]{kezu}.
\begin{proof}

Let $g(y) = |f(x+y)|^p e^{-\alpha p \left<(y+x),x\right>}.$
Now use the mean value property to the log-subharmonic function $g$ (it is also subharmonic).
$$|g(0)| \le c_{p, \alpha}\int_{ \mathbb{R}^m} |g(y)| e^{-\frac{\alpha p}{2} |y|^2} dA(y).$$

Then we have
$$g(0)=|f(x)|^p e^{-\alpha p |x|^2}\le c_{p, \alpha}\int_{ \mathbb{R}^m} f^p(y+x)e ^{-\frac{\alpha}{2} p \left<(x+y), x\right>} e^{-\frac{\alpha p}{2} |y|^2} dA(y).$$

Therefore
$$|f(x)|^p e^{-{\alpha} p |x|^2}\le  c_{p, \alpha}\int_{ \mathbb{R}^m} f^p(y)e ^{-\alpha p \left<y,x\right>} e^{-\frac{\alpha p}{2} |y-x|^2} dA(y).$$

So $$|f(x)|^p e^{-\frac{\alpha p}{2} |x|^2}\le   c_{p, \alpha}\int_{ \mathbb{R}^m} f^p(y) e^{-\frac{\alpha p}{2} |y|^2} dA(y).$$
Now, to prove \eqref{shark}, we use the following inequality, which is also a consequence of the sub-mean value property of subharmonic functions. Let $B_1(x)=\{y\in\mathbb{R}^m: |y-x|<1\}$. Then for every subharmonic function $g$ we have
$$|g(0)|  \le \frac{n}{\omega_n} \int_{B_1(0)} |g(y)| dA(y).$$
Thus
\begin{equation}\label{balline}|g(0)| e^{-\frac{\alpha p}{2} }  \le \frac{n}{\omega_n} \int_{B_1(0)} |g(y)| e^{-\frac{\alpha p}{2} |y|^2} dA(y).\end{equation}



By applying the previous inequality for $g(y)= |f(x+y)|^p  e^{-\alpha p \left<(y+x),x\right>},$ we obtain from \eqref{balline} that

\[\begin{split} |f(x)|^p e^{-\alpha p |x|^2 }e^{-\frac{\alpha p}{2} }  &\le {\frac{n}{\omega_n}} \int_{B_1(0)} |f(x+y)|^p e^{-\alpha p \left<(y+x),x\right>} e^{-\frac{\alpha p}{2} |y|^2}dA(y)
\\&=\frac{n}{\omega_n}\int_{B_1(x)} |f(y)|^p e ^{-\alpha p \left<y,x\right>} e^{-\frac{\alpha p}{2} |y-x|^2} dA(y)\\&=
\frac{n}{\omega_n}e^{-\frac{\alpha p}{2} |x|^2 }\int_{B_1(x)} |f(y)|^p e^{-\frac{\alpha p}{2} |y|^2} dA(y).
 \end{split}\]
 Thus $$|f(x)|^p e^{-\frac{\alpha p}{2} |x|^2 }e^{-\frac{\alpha p}{2} }\le \frac{n}{\omega_n}\int_{B_1(x)} |f(y)|^p e^{-\frac{\alpha p}{2} |y|^2} dA(y).$$
 Since $f\in \mathcal{L}_\alpha^p$, it follows that $$\lim_{|x|\to \infty} \frac{n}{\omega_n}\int_{B_1(x)} |f(y)|^p e^{-\frac{\alpha p}{2} |y|^2} dA(y)=0.$$ This implies  \eqref{shark}.
\end{proof}

It follows from the following lemma that $\|f\|_\alpha$ is a norm on $\mathcal{F}_\alpha$. Theorem~\ref{fockBanah} is a direct application of the following lemma
\begin{lemma}

a) If $f,g\in \mathcal{F}_{\alpha}$, then $\|f+g\|_\alpha\le \|f\|_\alpha+\|g\|_\alpha$.

b)  For every $\alpha>0$ and  $f\in \mathcal{F}_{\alpha}$ and $x\in \mathbb{C}^m$ we have $|f(x)|e^{-\frac{\alpha}{2}|x|^2} \le \|f\|_\alpha$.

c) For every $\alpha>0$ and  $f\in \mathcal{F}_{\alpha}$, $\|f\|_{\alpha}=\sup_{x\in \mathbb{C}^n} \left(|f(x)|e^{-\frac{\alpha}{2}|x|^2}\right).$

\end{lemma}

\begin{proof}
Let us restrict ourselves to the case $n=1$. The general case is a trivial modification of this case.

a) Let $f, g\in \mathcal{F}_\alpha$. Then for every $\alpha>0$, $f, g\in \mathcal{F}^p_{\alpha}$ and by the triangle inequality for the norm in $\mathcal{F}^p_{\alpha}$ we obtain  \[\begin{split}\|f+g\|_\alpha&=\lim_{p\to \infty}\|f+g\|_{p, \alpha}\\&\le \lim_{p \to \infty}\|f\|_{p, \alpha}+\lim_{p \to \infty}\|g\|_{p, \alpha}\\&=\|f\|_\alpha+\|g\|_\alpha.\end{split}\]

b) This follows from Proposition~\ref{ulqin}.

c) It follows from \eqref{sharine} that $$|f(x)|e^{-\frac{\alpha }{2} |x|^2}\le \|f\|_{p,\alpha}.$$ By letting $p\to \infty$ we obtain $$|f(x)|e^{-\frac{\alpha }{2} |x|^2}\le \|f\|_{\alpha}.$$ Thus $$\mathrm{ess}\,\sup |f(x)|e^{-\frac{\alpha }{2} |x|^2}\le \|f\|_\alpha.$$
To prove the converse, fix an $R>0$ and assume first that $f=P$ is a polynomial. Then $$\|P\|^p_{p,\alpha}=\int_{|x|\le R}|P(x)|^{p}e^{-\frac{\alpha}{2} p |x|^2}dx+\int_{|x|>R}|P(x)|^{p}e^{-\frac{\alpha}{2} p |x|^2}dx.$$

Moreover for sufficiently large  $R$ $$I(R):=\int_{|x|>R}|P(x)|^{p}e^{-\frac{\alpha}{2} p |x|^2}dx\le c_P \int_{|x|>R}|z|^{n_P p}e^{-\frac{\alpha}{2} p |x|^2}dx$$ and the last expression is smaller than $\|F\|^p_{\infty, \alpha}$. In fact the last expression tends to zero as $R\to \infty$. Therefore

$$\|P\|_{p,\alpha}\le (\|P\|^p_{\infty,\alpha}  R^n \omega_n+\|P\|^p_{\infty, \alpha})^{1/p},$$ where $\omega_n$ is the meausre of the unit sphere. Thus $$\|P\|_\alpha=\lim_{p\to \infty} \|P\|_{p,\alpha}\le \|P\|_{\infty, \alpha}.$$ Thus if $f$ is a polynomial, then \begin{equation}\label{infp}\|f\|_\alpha=\|f\|_{\infty,\alpha}.\end{equation}

Further, if $f$ is not a polynomial and $\epsilon>0$ is arbitrary, then for $p=2$, there exists a polynomial $P$ so that $\|P-f\|_{p,\alpha}<\epsilon$. Moreover, $$\|f\|_\alpha\le \|P\|_\alpha+\|f-P\|_\alpha=\|P\|_{\infty,\alpha}+\|f-P\|_\alpha\le \|P\|_\alpha+\epsilon. $$ Since $\epsilon$ is arbitrary, we conclude that \eqref{infp} hold for every function $f\in \mathcal{F}_\alpha$.

\end{proof}
\begin{remark}
One can ask, given a holomorphic function $f$, when this $$\lim_{p \to 0}\|f\|_{\alpha , p}$$ exists.  The answer is that limit is infinity except for the case when $f\equiv const$, so we cannot produce a Hardy type space for holomorphic mappings in $\mathbb{C}^n$.
\end{remark}

\subsection*{Acknowledgments} After I proved my main results, Kehe Zhu directed me to the reference by Carlen \cite{carlen}, for which I am thankful. I am also grateful to Djordjije Vujadinovic for his valuable insights and discussions on the topic.

\end{document}